\documentclass[12pt]{article}

\usepackage{amsmath, amssymb}
\usepackage{amsthm}
\usepackage{mathrsfs}
\usepackage[margin=1in]{geometry}
\usepackage{ifpdf}
\ifpdf
  \usepackage[hidelinks]{hyperref}
\else
  \usepackage[dvipdfmx,hidelinks]{hyperref}
\fi

\numberwithin{equation}{section}
\allowdisplaybreaks[3]

\title{On Weakly Separable Polynomials in Skew Polynomial Rings}
\author{Satoshi Yamanaka\\[0.25em]
\small Department of Integrated Science and Technology\\
\small National Institute of Technology, Tsuyama College\\
\small 624-1 Numa, Tsuyama-shi, Okayama, 708-8509, Japan\\
\small \texttt{yamanaka@tsuyama.kosen-ac.jp}}
\date{}

\newtheorem{thm}{\quad Theorem}[section]
\newtheorem{lem}[thm]{\quad Lemma}
\newtheorem{cor}[thm]{\quad Corollary}

\newtheorem{exm}[thm]{\quad Example}

\theoremstyle{remark}
\newtheorem{rmk}[thm]{\quad Remark}

\everymath{\displaystyle}

\def\ds{\displaystyle}

\def\De{\Delta}

\def\a{\alpha}

\def\P{\Phi}

\def\le{\left}
\def\ri{\right}

\def\ss{}

\def\al{\alpha}
\def\be{\beta}

\begin{document}

\maketitle

\begin{abstract}
The notion of weakly separable extensions was introduced by N.~Hamaguchi and A.~Nakajima as a generalization of separable extensions.
The purpose of this article is to give a characterization of weakly separable polynomials in skew polynomial rings.
Moreover, we shall show the relation between separability and weak separability in skew polynomial rings of derivation type.
\end{abstract}

\vspace{0.5cm}
\noindent
{\small {\bf Note for the arXiv version.} This manuscript corresponds to the article ``On Weakly Separable Polynomials in Skew Polynomial Rings,'' published in {\it Math. J. Okayama Univ.} {\bf 64} (2022), 47--61. The version of record is available at \href{https://doi.org/10.18926/mjou/62795}{10.18926/mjou/62795}.}
\vspace{0.5cm}



\section{Introduction}

This paper is the continuation of the author's previous paper \cite{Y1}. 

Let $A/B$ be a ring extension with common identity $1$, and $M$ an $A$-$A$-bimodule. 
An additive map $\delta : A \to M$ is called a {\it $B$-derivation of $A$ to $M$} if 
$\delta(zw) = \delta(z)w + z\delta(w)$  for any $z,w \in A$  
and $\delta(B)=\{0\}$. 
Moreover,  a $B$-derivation $\delta$ of $A$ to $M$ 
 is called {\it inner} if there exists $m \in M$ such that $\delta(z)=mz -zm$ for any $ z \in A$.  
We say that  $A/B$ is {\it separable}
if the $A$-$A$-homomorphism of $A \otimes_BA$ onto $A$
defined by $\sum_j  z_j \otimes w_j \mapsto \sum_j z_j w_j$ ($z_j, w_j \in A$) splits. 
It is well known that $A/B$ is separable if and only if   
every $B$-derivation of $A$ to $N$ is inner for any $A$-$A$-bimodule $N$ (cf. \cite[Satz 4.2]{E}). 
 $A/B$ is called {\it weakly separable} if every $B$-derivation of $A$ to $A$ is inner. 
The notion of weakly separable extensions was  introduced by N. Hamaguchi and A. Nakajima 
as a generalization of  separable extensions (cf. \cite{HN}). 
Obviously, a separable extension is weakly separable.

Throughout this article,  let $B$ be a ring, 
$\rho$ an automorphism of $B$, and 
$D$ a $\rho$-derivation of $B$ (i.e.  $D$ is an additive endomorphism of $B$ such that 
$D(\alpha \beta) =D(\alpha) \rho(\beta) + \alpha D(\beta)$ for any $\alpha$, $\beta \in B$). 
By $B[X;\rho, D]$ we denote  the skew polynomial ring in which
the multiplication is given by  $\alpha X = X\rho(\alpha) + D(\alpha)$ for any $\alpha \in B$. 
We write $B[X; \rho] = B[X;\rho, 0]$ and $B[X; D] = B[X; 1, D]$. 
Moreover, by $B[X;\rho, D]_{(0)}$   
we denote the set of all monic polynomials $f$ 
in $B[X;\rho,D]$  such that $fB[X;\rho,D]=B[X;\rho,D]f$.  
For each polynomial $f \in B[X;\rho,D]_{(0)}$, 
the quotient ring 
$B[X;\rho,D]/fB[X;\rho,D]$ is a free ring extension of $B$. 
A polynomial $f$ in   
$B[X;\rho,D]_{(0)}$ is called {\it separable} (resp. {\it weakly separable}) 
 in 
$B[X;\rho,D]$ if   
$B[X;\rho,D]/fB[X;\rho,D]$ is  separable  
(resp. weakly separable) over $B$. 

Let $B^\rho = \{ \alpha \in B \, | \, \rho(\alpha) = \alpha \}$. 
In the previous paper \cite{Y1}, we studied  weakly separable polynomials over rings.  
In particular, we showed a necessary and sufficient condition 
for a polynomial $f \in B[X;\rho]_{(0)} \cap B^\rho[X]$  
(resp. a $p$-polynomial $f \in B[X;D]_{(0)}$ with a prime number $p$) 
to be weakly separable in $B[X;\rho]$ (resp. $B[X;D]$) (cf. \cite[Theorem 3.2 and  Theorem 3.8]{Y1}). 
The purpose of this paper is to give some improvements  and generalizations  of our results 
for the general skew polynomial ring $B[X;\rho,D]$.
In section 2, we shall mention briefly on  
some properties for polynomials in $B[X;\rho,D]_{(0)}$. 
In section 3, 
we shall give  a necessary and sufficient condition for 
a  polynomial $f$ in $B[X;\rho,D]_{(0)} \cap B^\rho[X]$ 
to be weakly separable in $B[X;\rho,D]$.    
Moreover, we shall show the relation between  separability and  weak separability  
in $B[X;D]$.

\ss
\section{Polynomials in $B[X;\rho,D]_{(0)}$}

In this section, we shall mention briefly on polynomials in $B[X;\rho,D]_{(0)}$. 
We  inductively define additive  endomorphisms  $\P_{[i,j]}^{}$ 
($0 \leq  j \leq i$) of $B$  as follows: 
\begin{align*}
\P_{[i,j]}
= \le\{ \begin{array}{cl}
1 \  (={\rm the \ identity \ map}) & (i=j=0)\\
D^i &  (j=0,  \ i \geq 1)\\
\rho^i &  (i =j \geq 1)\\
 \rho \Phi_{[i-1,j-1]} + D \Phi_{[i-1,j]}   & (i \geq 2, \ 1 \leq j \leq i-1)
\end{array}
\ri.
\end{align*}
First we shall state the following. 

\ss
\begin{lem}\label{lem00}
For any $\al \in B$, there holds 
\begin{align*}
\alpha X^{i} = \sum_{j=0}^{i} X^j \P_{[i,j]} (\alpha) \ \ ( i \geq 0).
\end{align*}
\end{lem}
\begin{proof}
We shall show it by induction. 
It is true when $i=0$. 
Let $\alpha$ be arbitrary element  in $B$ and assume that 
it is true when $i \geq 0$. 
We have then   
\begin{align*}
\alpha X^{i+1} &= \alpha X^i \cdot X\\
&= \left( \sum_{j=0}^{i} X^j \P_{[i,j]} (\alpha)\right) X\\
&= \sum_{j=0}^{i} X^j \left( X \rho \P_{[i,j]} (\alpha) + D \P_{[i,j]} (\alpha) \right)\\
&= \sum_{j=0}^{i} X^{j+1} \rho \P_{[i,j]} (\alpha) 
+ \sum_{j=0}^{i} X^{j} D \P_{[i,j]} (\alpha)\\
&=\sum_{j=1}^{i+1} X^{j} \rho \P_{[i,j-1]}(\alpha) 
+ \sum_{j=0}^{i} X^{j} D \P_{[i,j]}(\alpha)\\
&= X^{i+1} \rho \P_{[i,i]} (\alpha) + 
 \sum_{j=1}^{i} X^{j}  \left( \rho \P_{[i,j-1]} + D \P_{[i,j]} \right)(\alpha)
 + D \P_{[i,0]} (\alpha)\\
&= X^{i+1} \P_{[i+1,i+1]} (\alpha) + 
 \sum_{j=1}^{i} X^{j}  \P_{[i+1,j]}(\alpha)
 + \P_{[i+1,0]} (\alpha)\\
 &= \sum_{j=0}^{i+1} X^{j}  \P_{[i+1,j]} (\alpha).
\end{align*}
This completes the proof. 
\end{proof}


\ss
\begin{lem}\label{lemc1} 
Let $f$ be a monic polynomial in $B[X;\rho,D]$ of the form 
 $f=\sum_{i=0}^m X^i a_i$ $( m \geq 1, a_m=1)$.  
Then $f$ is in  $B[X;\rho,D]_{(0)}$ 
if and only if 
\begin{enumerate}
\item $\ds{ a_j \rho^m(\alpha)= \sum_{i=j}^{m} \P_{[i,j]} (\alpha) a_{i}}$
  for any $\al \in B$   $( 0 \leq j \leq m-1)$. 
\item $\ds 
D(a_i)= \le\{ \begin{array}{cl}
a_{i-1}-\rho({a_{i-1}})+a_i \big( \rho(a_{m-1})-a_{m-1} \big)  & (1 \leq i \leq m-1)\\
a_0 \big( \rho(a_{m-1})-a_{m-1} \big)  &  (i=0)
\end{array}\ri.$.
\end{enumerate}
\end{lem}
\begin{proof} 
Let  $f=\sum_{i=0}^m X^i a_i$ $(m \geq 1, a_m=1)$ be in  $B[X;\rho,D]$ and 
$\alpha$   arbitrary element in $B$. 
As was shown in  \cite[Lemma 1.1]{I1}, 
$f$ is in $B[X;\rho,D]_{(0)}$ if and only if $\alpha f = f \rho^m(\alpha)$ 
and $Xf = f(X- \big(\rho(a_{m-1}) -a_{m-1}) \big)$.  
It follows from Lemma \ref{lem00} that 
\begin{align*}
\alpha f =
\sum_{i=0}^{m} \alpha X^i a_i 
=\sum_{i=0}^m \left( \sum_{j=0}^{i}  X^j \Phi_{[i,j]} (\alpha) \right) a_i =\sum_{j=0}^m X^j \left(  \sum_{i=j}^m \Phi_{[i,j]} (\alpha) a_i  \right).
\end{align*}
Noting that $f \rho^{m}(\alpha) = \sum_{j=0}^{m-1} X^j a_j \rho^{m}(\alpha)$, 
the equation  
$\alpha f =f \rho^m(\alpha)$ implies the condition (1), and conversely. 
Next we see that 
\begin{align*}
&f \Big( X- \big( \rho(a_{m-1})-a_{m-1} \big) \Big)\\
&=  \sum_{i=0}^{m} X^{i} a_i X - \sum_{i=0}^{m}X^i a_i \big( \rho(a_{m-1})-a_{m-1}\big)\\
&= \sum_{i=0}^{m} X^{i} \big( X\rho(a_i) +D(a_i) \big)
- \sum_{i=0}^{m}X^i a_i \big( \rho(a_{m-1})-a_{m-1} \big) \\
&= \sum_{i=0}^{m} X^{i+1} \rho(a_i)  + \sum_{i=0}^{m} X^{i} D(a_i)
- \sum_{i=0}^{m}X^i a_i \big( \rho(a_{m-1})-a_{m-1} \big) \\
&= 
\sum_{i=1}^{m+1} X^i \rho (a_{i-1}) + 
\sum_{i=0}^{m} X^i \Big( D(a_i) -a_i \big( \rho(a_{m-1})-a_{m-1} \big)   \Big)\\
&=   X^{m+1} + X^m a_{m-1} + \sum_{i=1}^{m-1} X^i \Big( \rho(a_{i-1}) + D(a_i) -a_i \big( \rho(a_{m-1})-a_{m-1} \big)   \Big)\\
& \ \ \ \, + D(a_0) -a_0 \big( \rho(a_{m-1})-a_{m-1} \big). 
\end{align*}
Noting that $Xf = \sum_{i=0}^{m}X^{i+1} a_i = X^{m+1} + X^m a_{m-1} + \sum_{i=1}^{m-1} X^i a_{i-1}$, 
the  equation $Xf = f(X- \big(\rho(a_{m-1}) -a_{m-1}) \big)$ implies  that 
\[
\le\{ \begin{array}{l}
a_{i-1} = \rho(a_{i-1}) + D(a_i) -a_i \big( \rho(a_{m-1})-a_{m-1} \big) 
 \ \ (1 \leq i \leq m-1) \\
0= D(a_0) -a_0 \big( \rho(a_{m-1})-a_{m-1} \big) 
\end{array} \ri. .
\]
Hence we have the condition (2). The converse is obvious. 
\end{proof}
\ss

Recall that $B^\rho =\{ \a \in B \, | \, \rho(\a) = \a\}$. 
In addition, let  
$B^D = \{ \a \in B \, | \, D(\a) = 0\}$, 
$B^{\rho,D} = B^\rho \cap B^D$,  
and $C(B^{\rho,D})$ the center of $B^{\rho,D}$. 
We have then the following. 

\ss
\begin{cor}\label{cor01}
Let $f=\sum_{i=0}^m X^i a_i$ $(m \geq 1, a_m=1)$ be in $B[X;\rho,D]_{(0)}$. 
If $f \in B^\rho[X]$ then $f \in C(B^{\rho,D})[X]$.  
\end{cor}
\begin{proof}
Let $f=\sum_{i=0}^m X^i a_i$ $(m \geq 1, a_m=1)$ be in 
$B[X;\rho,D]_{(0)}$ and assume that $f \in B^\rho[X]$. 
Then, by Lemma \ref{lemc1} (2), 
 we have  
\begin{align*}
D(a_i)&=a_{i-1}-\rho({a_{i-1}})+a_i \big( \rho(a_{m-1})-a_{m-1} \big)\\
&=a_{i-1}-{a_{i-1}}+a_i \big( a_{m-1}-a_{m-1} \big)\\
&=0 \  \ \ (1 \leq i \leq m-1), \\
D(a_0) &= a_0 \big( \rho(a_{m-1})-a_{m-1} \big)\\ 
&=a_0 \big( a_{m-1}-a_{m-1} \big)\\
&=0.
\end{align*} 
Hence $a_i \in B^{D}$, that is, $a_i \in B^{\rho,D}$ ($0 \leq i \leq m-1$).  
Let $\beta$ be arbitrary element in $B^{\rho,D}$. 
It is clear  that $\P_{[i,j]}(\beta) =
\begin{cases}
\beta & (i=j )\\
0 &  ( i > j)
\end{cases} $. 
Therefore it follows from Lemma \ref{lemc1} (1) that 
\[
 a_j \beta = a_j \rho^m(\beta)
=\sum_{i=j}^{m} \P_{[i,j]} (\beta) a_{i} =  \beta a_j  \ \ (0 \leq j \leq m-1). 
\] 
Thus  $a_j \in C(B^{\rho,D})$ $(0 \leq j \leq m-1)$.  
\end{proof}

\ss
\section{Weakly separable polynomials in $B[X;\rho,D]$}

The conventions and notations employed in the preceding section 
will be used in this section. 
Throughout this section,  let $R=B[X;\rho,D]$, 
$R_{(0)}=B[X;\rho,D]_{(0)}$, and  $f$ a monic polynomial in $R_{(0)} \cap B^\rho[X]$ 
of the form  $f=\sum_{i=0}^m X^i a_i$ ($m \geq 1$, $a_m =1$).   
Note that  $f$ is in $C(B^{\rho,D})[X]$ by Corollary \ref{cor01}.  
We shall use the following conventions: 
\begin{itemize}
\item $R_1 
= \{ g \in R \, | \, \alpha  g = g \rho(\alpha)  \  
( \forall  \alpha \in B )\}$ 
\item $A = R/fR$ (the quotient ring of $R$ modulo $fR$)
\item $x = X +f R \in A$ (i.e. $\{ 1,x, x^2, \cdots , x^{m-1} \}$ is a free $B$-basis of $A$ and $x^m =-\sum_{j=0}^{m-1} x^j a_j$) 
\item $I_x $ is an inner derivation of $A$ by $x$ (i.e. $I_x(z) = zx-xz$ ($\forall z \in A$))
\item $C(A) $ is the center of $A$
\item $A_{k}
=\{ u \in A\, | \, \alpha u = u \rho^k (\alpha)  \  
(\forall \alpha \in B)\} $ \ ($k \in \mathbb{Z}$)
\item $V= A_0 = \{ z \in A \, | \, \a z = z \a \ (\forall \a \in B) \} $ 
(i.e. $V$ is the centralizer of $B$ in $A$)
\end{itemize}
Moreover, we define 
polynomials 
$Y_j \in R \cap  C(B^{\rho,D})[X]$ ($0 \leq j \leq m-1$) as follows: 
\begin{align*}
Y_0 &= X^{m-1} + X^{m-2}a_{m-1} + \cdots + Xa_2 + a_1,\\
Y_1 &= X^{m-2} + X^{m-3}a_{m-1} + \cdots + Xa_3 + a_2,\\
\cdot  &\cdot  \cdot  \cdot  \cdot \\
Y_j &= X^{m-j-1} + X^{m-j-2}a_{m-1} + \cdots + Xa_{j+2} + a_{j+1} \left(=\sum_{k=j}^{m-1} X^{k-j} a_{k+1} \right),\\
\cdot  &\cdot  \cdot  \cdot  \cdot \\
Y_{m-2} &= X + a_{m-1},\\
Y_{m-1} &= 1.
\end{align*} 
%
It is obvious that 
\begin{align}\label{eqXYj}
XY_j = 
\le\{\begin{array}{cl}
Y_{j-1}-a_j & (1 \leq j \leq m-1)\\
 f-a_0 & (j=0)
\end{array} \ri. .
\end{align}
The polynomials $Y_j$ ($0 \leq j \leq m-1$) 
were introduced by Y. Miyashita  to characterize  separable polynomials 
in $B[X;\rho,D]$ (cf. \cite{M}). 
Now let $y_j = Y_j + fR \in A$ ($0 \leq j \leq m-1$). 
Since the equality (\ref{eqXYj}), the following lemma is obvious.

\ss
\begin{lem}\label{lemxyj}
\[
xy_j = 
\le\{\begin{array}{cl}
y_{j-1}-a_j & (1 \leq j \leq m-1)\\
 -a_0 & (j=0)
 \end{array}\ri..
\]
\end{lem}

\ss
So we define a map $\tau : A \to A$ by
\[
\tau \left( z \right) = 
\sum_{j=0}^{m-1} y_j z x^j  \ \ (z \in A).
\]
Obviously, $\tau$ is a $C(A)$-$C(A)$-endomorphism of $A$. 
By making use of $\tau$, separable polynomials in $R$ are characterized 
as follows:

\ss
\begin{lem}\label{propM} {\rm (\cite[Theorem 1.8]{M} or \cite[Theorem 1.3]{YI1}) } 
$f$ is  separable  in $R$ if and only if 
there exists $u \in A_{1-m}$ $($that is, $\rho^{m-1}(\alpha) u = u \alpha$ for any $\alpha \in B)$ such that 
 $\tau(u) = 1$. 
\end{lem}
\ss

\begin{rmk}
Needles to say, each $a_i$  ($0 \leq i \leq m-1$) satisfies that 
$a_i x = x a_i$ and $a_i u = u a_i$  for any $u \in A_k$ ($k \in \mathbb{Z}$). 
In particular,  we see that $y_i x = x y_i$ ($0 \leq i \leq m-1$). 
\end{rmk}
\ss

First we shall prove the following lemma concerning 
the inner derivation $I_x$ and 
the $C(A)$-$C(A)$-endomorphism $\tau$. 

\ss
\begin{lem}\label{lemT} 
\begin{enumerate}
\item $I_x (A_k) \subset  {\rm Ker}(\tau)$ 
for any integer $k$. 
\item $I_x (V) \subset  {\rm Ker}(\tau) \cap A_1$. 
\end{enumerate}
\end{lem}

\begin{proof} 
(1) Let $k$ be   arbitrary integer and $u \in A_k$.  
We obtain  then 
\begin{align*}
\tau(I_x(u)) &=\tau(ux -xu)\\
&= \sum_{j=0}^{m-1} y_j (ux -xu)x^j\\
&= \le(\sum_{j=0}^{m-1} y_j u x^j \ri)x - x \le(\sum_{j=0}^{m-1} y_j u x^j \ri)\\
&= \tau(u)x - x \tau(u). 
\end{align*}
Therefore it suffice to prove that $\tau(u) x = x \tau(u)$. 
By Lemma \ref{lemxyj}, we have 
\begin{align*}
x \tau (u) &= x \left( \sum_{j=0}^{m-1} y_j u x^j \right)\\
&= xy_0 u +  \sum_{j=1}^{m-1} xy_j u x^j\\
&= -a_0  u +  \sum_{j=1}^{m-1} (-a_j + y_{j-1} ) u x^j\\
&= -u a_0 - u \sum_{j=1}^{m-1} x^ja_j + \sum_{j=1}^{m-1} y_{j-1}u x^j \\
&= u \left( -\sum_{j=0}^{m-1} x^ja_j\right) + \sum_{j=0}^{m-2} y_{j}u x^{j+1}\\ 
&= ux^m+\left( \sum_{j=0}^{m-2} y_{j}u x^{j}\right)x\\
&= (y_{m-1} u x^{m-1}) x +\left( \sum_{j=0}^{m-2} y_{j}u x^{j}\right)x\\
&= \left(\sum_{j=0}^{m-1} y_{j}u x^{j}\right)x\\
&= \tau(u)x.
\end{align*}

(2) Since the condition (1), it suffice to show  that $I_x ( V ) \subset A_1$. 
For any $\a \in B$ and $u \in V$, we obtain 
\begin{align*}\alpha I_x(u) 
&= \alpha (ux-xu) \\
&= u \alpha x -\alpha x u \\
&=u(x\rho(\alpha) +D(\alpha)) -(x\rho(\alpha ) + D(\alpha ) )u \\
&=ux\rho(\alpha) +uD(\alpha) -xu\rho(\alpha ) - uD(\alpha )\\
&=(ux-xu) \rho(\alpha)\\
& =I_x(u)\rho(\alpha).
\end{align*} 
Thus  $I_x(V) \subset A_1$. 
\end{proof}


\ss
To show the subsequent lemma (Lemma \ref{lemDA}),  we need the following 
two  lemmas. 

\ss
\begin{lem}\label{lemR1}
Let $g_1$ be arbitrary element in $R$.  
We define $g_0 = 0$ and 
$g_{j+1} = g_j X + X^j g_1$ $(j \geq 1)$, inductively. 
\begin{enumerate}
\item $g_{i+k} = g_i X^k + X^i g_k$  $(i,k \geq 0)$. 
\item If $g_1 \in R_1$ then 
$\alpha g_{j} = \sum_{k=1}^j g_k \Phi_{[j,k]}(\alpha)$   $(j \geq 1)$
for any $\alpha \in B$.  
\end{enumerate}
\end{lem}

\begin{proof}
(1) Fix $i \geq 0$ and  we shall show it by induction for $k$. 
It is true when $k=0$. 
Assume that 
it is true when $k \geq 1$. 
So we obtain  
\begin{align*}
g_{i+k+1} &= g_{i+k} X + X^{i+k} g_1\\
&= \le(g_i  X^k + X^i g_k \ri) X + X^{i+k} g_1\\
&=g_i X^{k+1} + X^i g_k X + X^{i+k} g_1\\
&= g_i X^{k+1} + X^i \le( g_k X + X^k g_1 \ri)\\
&= g_i X^{k+1} + X^i g_{k+1}.
 \end{align*}
This completes the proof. 

(2) Let $g_1$ be  arbitrary element in $R_1$. 
We shall show it by  induction. 
It is true when $j=1$. 
Assume that 
it is true when  $j \geq 1$. 
Then, by Lemma  \ref{lem00}, we have 
\begin{align*}
\al g_{j+1} &= \al \le( g_j X + X^j g_1 \ri)\\
&= \al g_j X + \al  X^j g_1\\
&= \sum_{k=1}^j  g_k \P_{[j,k]} (\al) X 
+ \sum_{k=0}^j X^k  \P_{[j,k]} (\al) g_1\\
&= \sum_{k=1}^j  g_k \le( X  \rho \P_{[j,k]} (\al ) +  D \P_{[j,k]} (\al ) \ri)
  + \sum_{k=0}^j X^k g_1  \rho \P_{[j,k]} (\al )\\
&= \sum_{k=1}^j g_k X \rho \P_{[j,k]} (\al ) 
 +  \sum_{k=1}^j  g_k D \P_{[j,k]} (\al )
+ \sum_{k=0}^j X^k  g_1  \rho \P_{[j,k]} (\al )\\
&= \sum_{k=0}^j \le( g_k   X + X^k g_1 \ri)  \rho \P_{[j,k]} (\al )  + \sum_{k=1}^j  g_k  D \P_{[j,k]} (\al )\\
&= \sum_{k=0}^j g_{k+1}  \rho \P_{[j,k]} (\al ) 
+ \sum_{k=1}^j  g_k   D \P_{[j,k]} (\al )\\
&= \sum_{k=1}^{j+1} g_k   \rho \P_{[j,k-1]} (\al ) 
 + \sum_{k=1}^j  g_k    D \P_{[j,k]} (\al )\\
&= g_{j+1} \rho \P_{[j,j]} (\al ) 
  + \sum_{k=1}^{j}g_k  \le(  \rho \P_{[j,k-1]} + D \P_{[j,k]} \ri) (\al) \\
&= g_{j+1} \P_{[j+1,j+1]} (\al ) + 
\sum_{k=1}^{j} g_k \P_{[j+1,k]} (\al) \\
&= \sum_{k=1}^{j+1} g_k \P_{[j+1,k]} (\al).
\end{align*}
This completes the proof. 
\end{proof}

\ss
\begin{lem}\label{lemDR}
For any $g \in R_1$,   
there exists a $B$-derivation $\De$ of $R$ such that 
$\De(X) =g$.
\end{lem}

\begin{proof}
Let $g$ be  arbitrary element in $R_1$. 
We define $g_0=0$, $g_1=g$, and 
$g_{j+1} = g_{j} X + X^j g_1 $ ($j \geq 1$), inductively.
Moreover,  let $\De$ be  a right $B$-endomorphism of $R$  
 defined by $\De(X^j) = g_j$ ($j \geq 0$) 
(that is, 
$\De \le( \sum_j X^j c_j \ri) =\sum_j g_j c_j$ ($c_j \in B , j \geq 0$)). 
For any  $i,j \geq 1$ and $\alpha,\beta \in B$,   
it follows from Lemma \ref{lem00} and  Lemma \ref{lemR1} that   
\begin{align*}
\Delta(X^i \alpha X^j \beta) 
&= \Delta \left( X^i  \left( \sum_{k=0}^j X^k  \P_{[j,k]} (\alpha) \right) \beta \right)\\
&=  \Delta \left( \sum_{k=0}^j X^{i+k} \P_{[j,k]} (\alpha) \beta  \right)\\
&= \sum_{k=0}^j  g_{i+k}  \P_{[j,k]} (\alpha) \beta\\
&= \sum_{k=0}^j  \le( g_i  X^k  + X^i g_k  \ri) \P_{[j,k]} (\al) \be\\
&= g_i   \sum_{k=0}^j X^k \P_{[j,k]} (\al) \be 
+ X^i \sum_{k=1}^j g_k \P_{[j,k]} (\al) \be \\
&= g_i  \al X^j \be + X^i \al g_j  \beta\\
&= \De(X^i \al) X^j \be + X^i \al \De(X^j \be). 
\end{align*}
This implies that $\De( h_1 h_2 ) = \De(h_1) h_2 + h_1 \De(h_2)$ for any 
$h_1, h_2 \in R$, that is, $\De$ is a derivation of $R$. 
\end{proof}
\ss




Now we shall characterize $B$-derivations of $A$ as follows: 
\ss
\begin{lem}\label{lemDA}
If $\delta$ is a $B$-derivation of $A$ then 
$\delta(x) \in A_1 \cap {\rm Ker}  (\tau)$. 
Conversely, if $u \in A_1 \cap {\rm Ker}(\tau)$  
then there exists a $B$-derivation $\delta$ of $A$ such that $\delta(x) = u$.  
\end{lem}

\begin{proof}
Let $\delta$ be a $B$-derivation of $A$. 
We have $\alpha \delta(x) = \delta(\alpha x ) = \delta(x \rho(\alpha) + D(\alpha)) = \delta(x) \rho(\alpha)$ 
 for any $\alpha \in B$, and hence $\delta(x) \in A_1$. 
An easy induction shows that 
\begin{align*}
\delta(x^{k+1}) =\sum_{j=0}^k x^{k-j} \delta(x) x^j \ \ (k \geq 0).
\end{align*} 
Then, since $0 = \sum_{k=0}^m x^k a_k$ and $y_j = \sum_{k=j}^{m-1}x^{k-j}a_{k+1}$, we  see that 
\begin{align*}
0 &= \delta \left( \sum_{k=0}^m x^k a_k \right)\\
&= \sum_{k=1}^m  \delta(x^k) a_k\\
&= \sum_{k=0}^{m-1} \delta (x^{k+1} )  a_{k+1} \\
&= \sum_{k=0}^{m-1} \left( \sum_{j=0}^k x^{k-j} \delta(x) x^j \right) a_{k+1}  \\
&= \sum_{j=0}^{m-1} \left( \sum_{k=j}^{m-1} x^{k-j} a_{k+1} \right) \delta(x) x^j\\
&= \sum_{j=0}^{m-1} y_j \delta(x) x^j\\
&= \tau(\delta(x)).
\end{align*}
Therefore $\delta(x) \in {\rm Ker}(\tau)$. 

Conversely, assume that $u  \in A_1 \cap {\rm Ker}(\tau)$. 
Let $u_0$ be in $R$ such that $u=u_0+ fR $ and ${\rm deg} ( u_0) <m$.  
Obviously, $u_0 \in R_1$ because $u \in A_1$. 
Hence, by Lemma \ref{lemDR}, there exists a $B$-derivation $\Delta$ of $R$ such that $\Delta(X) = u_0$. 
Since $u \in {\rm Ker}(\tau)$, we have
\[
0= \tau(u)= \sum_{j=0}^{m-1} y_j u x^j = \sum_{j=0}^{m-1} Y_j u_0  X^j + fR.  
\]
This means that $\sum_{j=0}^{m-1} Y_j \Delta(X) X^j\in fR$. 
So we obtain  
\begin{align*}
\Delta(f) 
&=\sum_{k=1}^m \Delta(X^k) a_{k}\\
&= \sum_{k=0}^{m-1} \Delta(X^{k+1}) a_{k+1}\\
&= \sum_{k=0}^{m-1} \left( \sum_{j=0}^k X^{k-j} \Delta(X) X^j \right) a_{k+1}\\
&= \sum_{j=0}^{m-1} \left(\sum_{k=j}^{m-1} X^{k-j} a_{k+1}\right) \Delta(X) X^j\\ 
&= \sum_{j=0}^{m-1} Y_j \Delta(X) X^j\\ 
& \in fR. 
\end{align*}
This implies  that $\Delta(f g) = \Delta(f) g + f \Delta(g) \in fR$ for any $g \in R$, 
namely,   $\Delta(fR) \subset fR$.  
Thus there exists a $B$-derivation $\delta$ of $A$ such that 
$\delta(x)  =u$ which is naturally induced by $\Delta$. 
\end{proof}
\ss

Now we shall state the following theorem which is a generalization of 
\cite[Theorem 3.2] {Y1} and \cite[Theorem 3.8]{Y1}. 
\ss
\begin{thm}\label{thmWS1}
$f$ is weakly separable in $R$ if and only if 
\[
A_1 \cap {\rm Ker}(\tau)  =  I_x(V).
\] 
\end{thm}

\begin{proof}

Note that 
$I_x(V) \subset {\rm Ker}(\tau) \cap A_1$ by Lemma \ref{lemT} (2).   

Assume that $f$ is weakly separable in $R$, that is, every $B$-derivation of $A$ is inner. 
Let $u \in A_1 \cap {\rm Ker}(\tau)$. 
By Lemma \ref{lemDA}, there exists a $B$-derivation $\delta$ of $A$ such that $\delta(x) =u$. 
Since $\delta$ is inner, we have $u=\delta(x) = vx-xv$ for some fixed element $v \in A$. 
In particular, we see that $v \in V$ because $0=\delta(\alpha)=v \alpha -\alpha v$ 
for any  $\alpha \in B$. 
We have then  $u=\delta(x) \in I_x(V)$, namely, $A_1 \cap {\rm Ker}(\tau) \subset I_x(V)$.  
Therefore  $A_1 \cap {\rm Ker}(\tau) = I_x(V)$ by Lemma \ref{lemT} (2).   

Conversely, assume that $A_1 \cap {\rm Ker}(\tau) = I_x(V)$, 
and let $\delta$ be a $B$-derivation of $A$. 
By Lemma \ref{lemDA}, we see that $\delta(x) \in A_1 \cap {\rm Ker}(\tau)= I_x(V)$. 
Hence $\delta(x) = vx -xv$ for some $v \in V$. 
An easy induction shows that 
\[
\delta(x^j) = v x^j -x^j v \ \  (j \geq 0). 
\]
So, for any $z=\sum_{j=0}^{m-1} x^j c_j \in A$ ($c_j \in B$), 
we have 
\begin{align*}
\delta(z) &= \delta\left( \sum_{j=0}^{m-1} x^j c_j \right)\\
&= \sum_{j=0}^{m-1} \delta(x^j) c_j \\
&= \sum_{j=0}^{m-1} ( v x^j -x^j v) c_j \\
&= v \sum_{j=0}^{m-1}x^jc_j - \sum_{j=0}^{m-1}x^jc_j v\\
&= v z - zv. 
\end{align*}
Therefore $\delta$ is inner, and hence $f$ is weakly separable in $R$. 
\end{proof}
\ss

In virtue of Theorem \ref{thmWS1}, we have the following. 
\ss 
\begin{cor}\label{corWS}
$f$ is weakly separable in $R$ if and only if 
the following sequence of $C(A)$-$C(A)$-homomorphisms is exact: 
\[
V \overset{I_x}{\longrightarrow} A_1 \overset{\tau}{\longrightarrow}  A.
\] 
\end{cor}

\begin{proof} 
It is obvious by Theorem \ref{thmWS1}. 
\end{proof}
\ss

\begin{rmk}
There always holds ${\rm Ker} (I_x : V \to A_1) = C(A)$. 
Hence  
$f$ is weakly separable in $R$ if and only if 
the following sequence of $C(A)$-$C(A)$-homomorphisms is exact: 
\[
0 \longrightarrow C(A)  \overset{{\rm inclusion}}{\longrightarrow}  
V \overset{I_x}{\longrightarrow} A_1 \overset{\tau}{\longrightarrow}  A.
\] 
\end{rmk}
\ss

Concerning the relation between  separability and  weak separability in $B[X;D]$, 
we shall  state the following  theorem which is an improvement 
of \cite[Theorem 3.10]{Y1}. 

\ss
\begin{thm}\label{WS2} 
Let $R=B[X;D]$, $R_{(0)}=B[X;D]_{(0)}$, and  
$f$  a monic polynomial  in $R_{(0)}$ of the form 
$f= \sum_{i=0}^m X^i a_i$ $(a_m=1, m \geq 1)$. 
\begin{enumerate}
\item  $f$ is weakly separable in $R$ if and only if 
the following sequence of $C(A)$-$C(A)$-homomorphisms is exact: 
\[
V \overset{I_x}{\longrightarrow} V  \overset{\tau}{\longrightarrow} C(A).
\] 

\item $f$ is separable in $R$  if and only if 
the following sequence of $C(A)$-$C(A)$-homomorphisms is exact: 
\[
V \overset{I_x}{\longrightarrow} V \overset{\tau}{\longrightarrow} C(A) \longrightarrow 0. 
\] 
\end{enumerate}
\end{thm}

\begin{proof} 
Note that $A_k=V$ ($k \in \mathbb{Z}$) in this case.  
First we shall show that $\tau (V) \subset C(A) $. 
Let 
$\varphi$ be an $A$-$A$-homomorphism of $A \otimes_BA$ onto $A$
defined by $\sum_j z_j  \otimes w_j  \mapsto \sum_j z_j w_j$ ($z_j,w_j \in A$) 
and $(A \otimes_B A)^A = \{ \mu \in A \otimes_B A \, | \, z \mu = \mu z \ (\forall z \in A)\} $.  
It is obvious that $\varphi \le( (A \otimes_B A)^A \ri) \subset C(A)$.  
Let $v$ be  arbitrary element in $V$. 
As was shown in \cite[Lemma 3.1]{YI1}, 
we have already known that 
$\sum_{j=0}^{m-1} y_j v \otimes x^j \in (A \otimes_B A)^A$,  
and hence 
\[
\tau(v) = \sum_{j=0}^{m-1} y_j v x^j 
=\varphi \left( \sum_{j=0}^{m-1} y_j v \otimes x^j \right) \in C(A).
\] 
Therefore  $\tau(V) \subset C(A)$. 

(1) It is obvious by Corollary \ref{corWS}. 

(2)  If $f$ is separable in $R$  then $f$ is always weakly separable in $R$, 
and therefore it suffices to show that $\tau(V) = C(A)$.  
By Lemma \ref{propM},  $f$ is separable in $R$ if and only if 
there exists $v \in V$ such that $\tau(v)=1$.  
This means that 
$\tau(V) = C(A)$ because $\tau $ is a $C(A)$-$C(A)$-homomorphism.  
\end{proof}

\begin{exm}
{\rm 
Let $B=\le[ \begin{matrix}
\mathbb{Z} & \mathbb{Z}\\
0 & \mathbb{Z}
\end{matrix}\ri]
$ (the upper triangular matrix over $\mathbb{Z}$), 
$D$ a derivation of $B$ defined by 
$D\le(\le[ \begin{matrix}
b_1 & b_2\\
0 & b_3
\end{matrix}\ri] 
\ri)
=\le[ \begin{matrix}
0 & b_2\\
0 & 0
\end{matrix}\ri]$ ($b_1,b_2,b_3 \in \mathbb{Z}$), 
$R=B[X;D]$, and $R_{(0)}=B[X;D]_{(0)}$.  
We put here $a=\le[ \begin{matrix}
3 & 0\\
0 & 1
\end{matrix}\ri] \in B$ and 
$f=X^2+Xa + a \in R$.  
It is easy to see that $\al f = f\al$ for any $\al \in B$ and 
$Xf = fX$, and hence $f \in R_{(0)}$. 
Now let $A=R/fR$ and  $x = X+fR$. 
One easily see  that   
\[
V=C(A)=\le\{ x \le[ \begin{matrix}
s & 0\\
0 & s
\end{matrix}\ri] + 
\le[ \begin{matrix}
s+t & 0\\
0 & t
\end{matrix}\ri]  \, 
\bigg|  \, s,t \in \mathbb{Z} \ri\}. 
\]
Let $v = xb+c$ be  arbitrary element in  $V$ 
such that 
$b
=\le[ \begin{matrix}
s & 0\\
0 & s
\end{matrix}\ri]$, 
$c=\le[ \begin{matrix}
s+t & 0\\
0 & t
\end{matrix}\ri] $ ($s,t \in \mathbb{Z}$). 
Since $xb =bx$ and $xc = cx$,  we obtain 
\[
I_x(v) = vx -xv = (xb+c) x -x(xb+c) = 0. 
\]
Thus $I_x(V)=\{ 0\}$. 
Recall that $y_0 =x+a$ and $y_1=1$ in this case.  
We have then  
\begin{align*}
\tau(v) &= y_0 v + y_1 v x\\
&= (x+a) (xb+c) + (xb+c)x\\
&= x^2 b + x(ab+c) + ac + x^2 b +xc\\
&= x^2 2b + x(ab+2c) + ac\\
&= (-xa-a) 2b + x(ab+2c) + ac\\
&= x(2c-ab) + ac -2ab\\
&= x \le( \le[ \begin{matrix}
2(s+t) & 0\\
0 & 2t
\end{matrix}\ri] - 
\le[ \begin{matrix}
3 & 0\\
0 & 1
\end{matrix}\ri]
\le[ \begin{matrix}
s & 0\\
0 & s
\end{matrix}\ri]\ri) \\
&  \ \ \ \ \ 
+ \le[ \begin{matrix}
3 & 0\\
0 & 1
\end{matrix}\ri]
\le[ \begin{matrix}
s+t & 0\\
0 & t
\end{matrix}\ri]
-2 
\le[ \begin{matrix}
3 & 0\\
0 & 1
\end{matrix}\ri]
\le[ \begin{matrix}
s & 0\\
0 & s
\end{matrix}\ri]\\
&= x \le[ \begin{matrix}
2t-s & 0\\
0 & 2t-s
\end{matrix}\ri]
+ \le[ \begin{matrix}
3t - 3s & 0\\
0 & t -2s
\end{matrix}\ri].
\end{align*} 
So we see that  $s=t=0$ (i.e. $v=0$) if $v \in {\rm Ker}(\tau)$.  
Therefore ${\rm Ker}(\tau)\cap V =\{ 0 \} = I_x(V)$, 
and hence $f$ is weakly separable in $R$ by Theorem \ref{WS2} (1). 
However, it is obvious  that $\tau(V) \subsetneq C(A)$ (for example, 
there are no elements $u \in V$ such that $\tau(u) =
 \le[ \begin{matrix}
1 & 0\\
0 & 1
\end{matrix}\ri] \in C(A)$). 
Thus $f$ is not separable in $R$ by  Theorem \ref{WS2} (2). 
}
\end{exm}

\section*{Acknowledgement}

The author would like to thank the referee for his (her) valuable suggestions and comments.

%
%
%
%
%
%

\end{document}